\documentclass[sn-mathphys,Numbered]{sn-jnl}

\usepackage{graphicx}%
\usepackage{multirow}%
\usepackage{amsmath,amssymb,amsfonts}%
\usepackage{amsthm}%
\usepackage{mathrsfs}%
\usepackage[title]{appendix}%
\usepackage{xcolor}%
\usepackage{textcomp}%
\usepackage{manyfoot}%
\usepackage{booktabs}%
\usepackage{algorithm}%
\usepackage{algorithmicx}%
\usepackage{algpseudocode}%
\usepackage{listings}%
\usepackage{natbib}
\usepackage[utf8]{inputenc}
\usepackage[english]{babel}
\usepackage{amsmath}
\usepackage{amsfonts}
\usepackage{amssymb}
\usepackage{verbatim}
\usepackage{amsmath}
\usepackage{amsfonts}
\usepackage{subfig}
\usepackage{pgf}
\usepackage{tikz}   
\usepackage{lipsum}
\usepackage{hyperref}
\usepackage{graphicx}
\usepackage{enumerate}
\usepackage{xcolor}
\usepackage{natbib} 
\definecolor{blue3}{RGB}{1,180,255}

\newtheorem{theorem}{Theorem}
\newtheorem{proposition}[theorem]{Proposition}%
\newtheorem{definition}{Definition}%

\newtheorem{lemma}{Lemma}%

\begin{document}

\title[Article Title]{Online multidimensional dictionary learning }


\author[1]{\fnm{Ferdaous} \sur{Ait Addi}}\email{f.aitaddi.ced@uca.ac.ma}
\author[1]{\fnm{Abdeslem Hafid} \sur{Bentbib}}\email{a.bentbib@uca.ac.ma}

\author[2]{\fnm{Khalide}\sur{Jbilou}}\email{khalide.jbilou@univ-littoral.fr}

\affil[1]{\orgdiv{Laboratory LAMAI}, \orgname{Faculty of Science and Technology, University Cadi Ayyad }, \orgaddress{\street{Abdelkarim Elkhattabi},
\city{Marrakech}, \postcode{42 000}, \country{Morrocco}}}

\affil[2]{\orgdiv{Laboratoire LMPA }, \orgname{Université du Littoral Côte d’Opale}, \orgaddress{\street{Ferdinand Buisson}, \city{Calais Cedex}, \postcode{ 62228}, 
\country{France}}}


\abstract{Dictionary learning is a widely used technique in signal processing and machine learning that aims to represent data as a linear combination of a few elements from an overcomplete dictionary. In this work, we propose a  generalization of the dictionary learning technique using the t-product framework, enabling efficient handling of multidimensional tensor data. We address the dictionary learning problem through online methods suitable for tensor structures. To effectively address the sparsity problem, we utilize an accelerated Iterative Shrinkage-Thresholding Algorithm (ISTA) enhanced with an extrapolation technique known as Anderson acceleration. This approach significantly improves signal reconstruction results. Extensive experiments prove that our proposed method outperforms existing acceleration techniques, particularly in applications such as data completion.
These results suggest that our approach can be highly beneficial for large-scale tensor data analysis in various domains.}


\keywords{Online dictionary learning, tensors, basis pursuit,  t-product, completion.}


\pacs[MSC Classification]{15A69, 15A72,15A83}

\maketitle
\section{Introduction}
Sparse representations aim to express a dataset as a combination of only a few key components, or atoms, chosen from a larger set called a dictionary. Consider a signal \(y\in \mathbb{R}^{N}\). We can represent \(y\) as:

\[
y = Dx = d_1x_1 + d_2x_2 + \dots + d_Mx_M,
\]

where \(D = [d_1, d_2, \dots, d_M] \in \mathbb{R}^{N \times M}\) is our dictionary, with \(M > N\), and \(x = [x_1, \dots, x_M]\) is the sparse representation. Essentially, this means we’re expressing \(y\) as a sum of a few components from \(D\), and the goal is to keep \(x\) as sparse as possible, only a few of its elements should be non-zero.

The challenge here is finding the sparse representation, given the signal \(y\) and the dictionary \(D\). This boils down to solving an optimization problem:

\[
\min_x \| Dx - y \|_F^2 \quad \text{such that} \quad \| x \|_0 \leq K,
\]
where \(\| . \|_0\) counts the number of non-zero elements in \(x\), and \(K\) specifies the maximum number of dictionary atoms allowed to represent the signal.

A crucial part of solving this sparse representation problem is choosing the right dictionary \(D\) that leads to the sparsest possible representation. In earlier work, fixed dictionaries like the discrete cosine transform (DCT) or wavelet transform (WT) were commonly used. However, it turns out that learning the dictionary from the data often produces much better results.

Dictionary learning is the technique used to discover a dictionary from a dataset that leads to the most effective sparse representation. Suppose we have a set of signals \(\{ y_i \}_{i=1}^n \in \mathbb{R}^M\) and an overcomplete dictionary \(D \in \mathbb{R}^{M \times I}\). For each signal \(y_i\) (\(i = 1, 2, \dots, n\)), we can find a sparse representation in \(D\) such that:

\[
Dx_i = y_i \quad \text{with} \quad \| x_i \|_0 \leq K \quad \text{for} \quad i = 1, \dots, n.
\]
The goal of dictionary learning is to find the optimal dictionary \(D\) that allows for the sparsest representations of the signals in our dataset. This is done by solving the following optimization problem:

\[
\min_{D, x_i} \sum_{i=1}^n \| Dx_i - y_i \|_F^2 \quad \text{such that} \quad \| x_i \|_0 \leq K.
\]

In this paper, we aim to extend the dictionary learning model to handle multidimensional signals using tensors and the \(t\)-product operation. $T$-product is a tensor tensor product first defined in  \cite{kilmer}\cite{KILMER2011} as generalization of matrix matrix product Previous works \cite{soltani}\cite{zhang}, focused only on 2D signals. We expand their models to accommodate signals of arbitrary dimensions. Additionally, while those works  often employed batch processing (such as TKSVD), we  introduce two online methods \cite{liu}\cite{mairal} based on projected stochastic gradient descent to accelerate the learning process.\\
To solve the completion problem, we generalize an acceleration method known as Anderson acceleration\cite{brezinski}\cite{walker}, which speeds up the iterative soft-thresholding algorithm (ISTA). Our experiments show that this method outperforms other acceleration techniques.

This paper is organized as follows: Section \ref{sec2} introduces the notations and reviews relevant definitions and propositions. Section \ref{sec3} details the online   dictionary learning approach. Section \ref{sect4} introduces the generalized dictionary learning model and presents two  online learning methods, specifically the projected stochastic gradient descent algorithm and a second order algorithm based on the ISTA algorithm, which is used for online dictionary learning. Section \ref{MCO} demonstrates the application of the proposed dictionary learning methods in solving matrix completion problems additionally, it covers a generalization of Anderson Acceleration applied to ISTA in the sparse coding step. Finally, Section \ref{NumR} presents numerical results that showcase the performance of all the algorithms discussed, and Section \ref{sec8} concludes the paper.

\section{Notations and definitions}\label{sec2}

\subsection{Notations}
In what follows, matrices are denoted by uppercase letters $X$, vectors are denoted by lowercase letters $x$, and tensors are denoted by calligraphic letters $\mathcal{X}$. The $(i_{1},i_{2},\dots,i_{N})$ elements of the tensor $\mathcal{X}$ is denoted by $\mathcal{X}_{i_{1},i_{2},\dots,i_{N}}$, we denote the cardinality of a set of indices $\mathcal{I}$ by $|\mathcal{I}|$. We denote by $[I]$ the set of indices $\{1,2,\dots,I\}$.

The following section presents definitions that will be used throughout this work. 
\subsection{Definitions}
\begin{definition}\cite{Kolda2}\cite{Kolda}
	The inner product of two tensors $\mathcal{X}, \mathcal{Y} \in \mathbb{R}^{I_{1}\times I_{2}\times\dots \times I_{N}}$  is defined by
	\[ \langle\mathcal{X},\mathcal{Y}\rangle=\sum_{i_{1}=1}^{I_{1}}\sum_{i_{2}=1}^{I_{2}}\cdots\sum_{i_{N}=1}^{I_{N}} \mathcal{X}_{i_{1},i_{2},\dots,i_{N}}\mathcal{Y}_{i_{1},i_{2},\dots,i_{N}}.  \]
\end{definition}
The norm of a tensor $\mathcal{X} \in \mathbb{R}^{I_{1}\times I_{2}\times\dots \times I_{N}}$ is: 
\[ \Vert \mathcal{X} \Vert_{F}= \sqrt{\langle\mathcal{X},\mathcal{X}\rangle}. \]
\begin{definition} We define the \textbf{unfold} operation, denoted as \textbf{unfold(.)}, to transform an \( N \)-th order tensor \( \mathcal{A} \in \mathbb{R}^{I_{1}\times I_{2} \times \dots \times I_{N}} \) into an \( I_{1}I_{N}\times I_{2} \dots\times I_{N-1} \) block tensor as follows
		\[ unfold(\mathcal{A})=\left[\begin{array}{c}
			\mathcal{A}_{1} \\
			\mathcal{A}_{2} \\
			\vdots \\
			\mathcal{A}_{I_{N}}
		\end{array}\right]. \]
If $\mathcal{A}$ is an 3-order tensor, then  \textbf{unfold($\mathcal{A}$)} is a block vector. The operation that takes \textbf{unfold(.)} back to tensor form is the \textbf{fold(.)} command. Specifically, \textbf{fold(.,$I_{N}$)} takes an $I_{1}I_{N}\times I_{2}\times\dots\times I_{N-1}$ block tensor and returns an $I_{1}\times\dots\times I_{N}$ tensor. That is
\[\textbf{fold}(\textbf{unfold}(\mathcal{A}),I_{N}).\]
 \end{definition}
 We define the frontal slices of a N-th order tensor $\mathcal{X}\in \mathbb{R}^{I_{1} \times I_{2} \times I_{3} \times \dots \times I_{N}}$ as \cite{Bentbib}:
 \begin{equation}\label{frslices}
 	\mathcal{X}^{(p)}=\mathcal{X}(:,:,k_{3},k_{4},\dots,k_{N})~~ \text{with}~~ p=k_{3}+\sum_{i=4}^{N}(k_{i}-1)\prod_{s=3}^{i-1}I_{s}.
 \end{equation}
 
In \cite{kilmer} the t-product operation is defined as a product between two third-order tensors  $\mathcal{A}\in \mathbb{R}^{I_{1}\times I_{2}\times I_{3}}$ and  $\mathcal{B}\in \mathbb{R}^{I_{2}\times l\times I_{3}}$:
\[\mathcal{A}\ast_{t}\mathcal{B}=\textbf{fold}(circ(\mathcal{A}).\textbf{unfold}(\mathcal{B})),\]
such that $circ(\mathcal{A})$ is defined as the block circulant matrix created from the frontal slices of the tensor $\mathcal{A}$:

\[circ(\mathcal{A})=\left[	\begin{array}{ccccc}
	A_{1} & A_{I_{3}} & A_{I_{3}-1}& \dots & A_{2} \\
	A_{2} & A_{1} & A_{I_{3}} & \cdots & A_{3} \\
	\vdots & \ddots & \ddots & \ddots & \vdots \\
	A_{I_{3}} & A_{I_{3}-1} & \cdots & A_{2} & A_{1}
\end{array}\right],\]

where $A_{i}=\mathcal{A}(:,:,i)$ for $i=1\dots I_{3}$.
The matrix 	$unfold(\mathcal{B})$ is defined as a block vector that contains the  frontal slices of the tensor $\mathcal{B}$:
\[ unfold(\mathcal{B})=\left[\begin{array}{c}
	B_{1} \\
	B_{2} \\
	\vdots \\
	B_{I_{3}}
\end{array}\right]. \]
The generalized t-product was defined as an 
$N$-th order tensor in \cite{MARTIN}. The definition is as follows:

\begin{definition}\label{tproduct}\cite{MARTIN}
	Let consider $\mathcal{A}\in \mathbb{R}^{I_{1}\times I_{2}\times I_{3}\times\dots\times I_{N}}$ and  $\mathcal{B}\in \mathbb{R}^{I_{2}\times l\times I_{3}\times \dots \times I_{N}}$.  The t-product operation between the two tensors of order $(p \geq 3)$  is defined recursively as:
	\begin{equation}\label{eq1}
		\mathcal{A}\ast_{t}\mathcal{B}=fold(circ(\mathcal{A})\ast_{t} unfold(\mathcal{B})),
	\end{equation}
	
	such that the tensor product $\mathcal{A}\ast_{t}\mathcal{B}\in \mathbb{R}^{I_{1}\times l\times I_{3}\times\dots\times I_{N}}$.
\end{definition}
The right-hand side involves a t-product of \((N-1)\)-th order tensors, with each successive t-product operation reducing the order of the tensors by one.

\[\left[	\begin{array}{ccccc}
	\mathcal{A}_{1} & \mathcal{A}_{I_{N}} & \mathcal{A}_{I_{N}-1}& \dots & \mathcal{A}_{2} \\
	\mathcal{A}_{2} & \mathcal{A}_{1} & \mathcal{A}_{I_{N}} & \cdots & \mathcal{A}_{N} \\
	\vdots & \ddots & \ddots & \ddots & \vdots \\
	\mathcal{A}_{I_{N}} & \mathcal{A}_{I_{N}-1} & \cdots & \mathcal{A}_{2} & \mathcal{A}_{1}
\end{array}\right]\ast_{t} \left[\begin{array}{c}
	\mathcal{B}_{1} \\
	\mathcal{B}_{2} \\
	\vdots \\
	\mathcal{B}_{I_{N}}
\end{array}\right], \]
where $\mathcal{A}_{i}=\mathcal{A}(:,:,\cdots,i)$, and $\mathcal{B}_{i}=\mathcal{B}(:,:,\cdots,i)$ for $i=1,\dots I_{N}$.

However, while multiplying two tensors by expanding them into blocks of block
circulant matrices leads to a clean recursive algorithm, it is very
costly in terms of memory and computational speed. It is much faster to convert to
the Fourier domain for the multiplication \cite{MARTIN}.
Specifically if $\mathcal{A}\in \mathbb{R}^{n_{1}\times n_{2}\times n_{3}}$ and $F_{n_{3}}\in \mathbb{R}^{ n_{3}\times n_{3}}$ is the discrete Fourier transform (DFT) matrix, then
	$$
	\left(F_{n_3} \otimes I_{n_1}\right) \cdot \operatorname{circ}(\operatorname{unfold}(\mathcal{A})) \cdot\left(F_{n_3}^* \otimes I_{n_2}\right)=\left[\begin{array}{llll}
		D_1 & & & \\
		& D_2 & & \\
		& & \ddots & \\
		& & & D_{n_3}
	\end{array}\right].
	$$
Where " $\otimes$ " denotes the Kronecker product " $F$ " denotes the conjugate transpose of $F$ and "." means standard matrix product. Therefore, the t-product for third-order tensors can be computed by folding back to a tensor the following:
$$
\left(F_{n_3}^* \otimes I_{n_1}\right)\left(\left(F_{n_3} \otimes I_{n_1}\right) \cdot \operatorname{circ}(\operatorname{unfold}(\mathcal{A})) \cdot\left(F_{n_3}^* \otimes I_{n_2}\right)\right)\left(F_{n_3} \otimes I_{n_2}\right) \operatorname{unfold}(\mathcal{B}).
$$

See \cite{MARTIN} for the generalized case of the t-product.
The computation of the t-product operation between two third-order tensors is shown in Algorithm \ref{t-product}.
\begin{algorithm}[H]
	\caption{ t-product of real third-order tensors}
	\begin{algorithmic}[1] 
		
		\Require $\mathscr{A} \in \mathbb{R}^{\ell \times q \times n}, \mathscr{B} \in \mathbb{R}^{q \times p \times n}$.
		\Ensure $\mathscr{C}=\mathscr{A} * \mathscr{B} \in \mathbb{R}^{\ell \times p \times n}$.
		\State 
		 $\widehat{\mathscr{A}}=\operatorname{fft}(\mathscr{A},[], 3)$, and $\widehat{\mathscr{B}}=\mathrm{fft}(\mathscr{B},[], 3)$
		\For {$i=1,2$, to $\left[\frac{n+1}{2}\right]$} \\ 
		$\widehat{\mathscr{C}}^{(i)}=\widehat{\mathscr{A}}^{(i)} \widehat{\mathscr{B}}^{(i)}$
		\EndFor
		
	\For {$i=\left[\frac{n+1}{2}\right]+1$ to $n$} \\
	$\widehat{\mathscr{C}}^{(i)}=\operatorname{conj}\left(\widehat{\mathscr{C}}^{(n-i+2)}\right)$
	\EndFor
		\State  $\mathscr{C}=\operatorname{ifft}(\widehat{\mathscr{C}},[], 3)$
	\end{algorithmic}
	\label{t-product}
\end{algorithm} 

\begin{definition}\cite{mairal}
	Consider a third order tensor $\mathcal{A}\in \mathbb{R}^{I_{1}\times I_{2}\times I_{3}}$, then the transpose $\mathcal{A}^{T}\in \mathbb{R}^{I_{2}\times I_{1}\times I_{3}}$  is obtained by transposing each frontal slice and then reversing the order of transposed frontal slices 2 through $I_{3}$.
\end{definition}
The order $p$-transpose defined recursively based on definition  above:
\begin{definition}\cite{MARTIN}
	Consider a tensor $\mathcal{A}\in \mathbb{R}^{I_{1}\times I_{2}\times\dots \times I_{N}}$, then the transpose $\mathcal{A}^{T}\in \mathbb{R}^{I_{2}\times I_{1}\times \dots \times I_{N}}$  is obtained by transposing each $\mathcal{A}_{i}$ $i=1\dots I_{N}$ and then reversing the order of $\mathcal{A}_{i}$ 2 through $I_{N}$
	\[\mathcal{A}^{T}=fold\left(\left[\begin{array}{c}
		\mathcal{A}^{T}_{1}  \\
		\mathcal{A}^{T}_{I_{N}}\\
		\mathcal{A}^{T}_{I_{N}-1} \\
		\vdots \\
		\mathcal{A}^{T}_{2} 
	\end{array}\right]\right). \]
\end{definition}
\begin{definition}
	The $3$-order  identity tensor $\mathcal{I}_{nnl}\in \mathbb{R}^{n\times n\times l}$, is the tensor whose first frontal slice is the identity matrix and whose other frontal slices are all zeros.
\end{definition}
The $p$-order identity tensor is a generalization of the definition above.

\begin{definition}\cite{MARTIN} The \(p\)-order identity tensor \(\mathcal{I} \in \mathbb{R}^{n \times n \times l_{1} \times \dots \times l_{p-2}}\), \(p > 3\), is the tensor such that: 1. \(\mathcal{I}_{1}\) is the \((p-1)\)-order identity tensor. 2. For \(j = 2, \dots, l_{p-2}\), \(\mathcal{I}_{j}\) is the \((p-1)\)-order zero tensor, where \(\mathcal{I}_{j} = \mathcal{I}(:, :, \cdots, j)\)  for $j=2,\dots,l_{p-2}$.  \end{definition}
\begin{definition}
Consider a tensor $\mathcal{A}\in \mathbb{R}^{I_{1}\times I_{1}\times I_{3}\times \dots \times I_{N}}$, the tensor $\mathcal{A}$ is symmetric under the t-product if 
\[ \mathcal{A}^{T}=\mathcal{A}.\] 
\end{definition}

\begin{definition}\cite{Kilmer2013}
The tensor $\mathcal{A}\in \mathbb{R}^{I_{1}\times I_{1}\times\times I_{3} \dots \times I_{N}}$ is positive definite if each frontal slice $\hat{\mathcal{A}}^{(i)}$ in the fourier domain are Hermitian positive definite.
\end{definition}

\section{Online Dictionary Learning}\label{sect4}

Online dictionary learning is a computational technique commonly used in machine learning and signal processing to dynamically construct a dictionary of representative features directly from incoming data \cite{mairal}\cite{liu}. Unlike traditional methods, which process the entire dataset at once, online dictionary learning updates the dictionary incrementally as new data points are introduced. This approach is highly efficient for large datasets and allows for real-time adaptation to changes in the underlying data distribution.

In this framework, the following loss function is used to evaluate the quality of the learned dictionary:

\[
f(D; y) = \min_{x} \ell(x, D;y),
\]
where
\[
\ell(x, D;y) = \frac{1}{2}\|Dx - y\|_2^2 + \lambda \|x\|_1.
\]

Here, \(y\) represents a data point from the distribution. The term \(\|Dx - y\|_2^2\) measures the reconstruction error, while \(\lambda \|x\|_1\) enforces sparsity on the coefficients \(x\), balancing reconstruction accuracy with sparse representation.

The dictionary \(D\) is constrained by the set:
\[
C = \{ D \in \mathbb{R}^{M \times I} : \|d_j\|_2 = 1 \text{ for } j = 1, \dots, I \}.
\]
This constraint ensures that each column of \(D\) (denoted as \(d_j\)) has a unit \(\ell_2\)-norm, stabilizing the learning process and preventing trivial solutions.

The goal of dictionary learning is to solve the following optimization problem:
\[
\min_{D \in C} \frac{1}{N} \sum_{i=1}^{N} f(D; y_i),
\]
where each \(y_i\) is drawn randomly from the dataset.

In an online learning setting, where data points \(\{ y_i \}_{i=1}^{n}\) arrive in a streaming fashion, this problem can be addressed using a first-order method such as projected stochastic gradient descent. The dictionary \(D\) is updated iteratively as follows:

\begin{equation}
	D^{(t)} = \text{Proj}_C \left( D^{(t-1)} - \eta^{(t)} \nabla f(D^{(t-1)}; y^{(t)}) \right),
\end{equation}

where \(\eta^{(t)}\) is the learning rate, and \(\text{Proj}_C\) represents projection onto the constraint set \(C\).

Alternatively, a second-order method can be used to update the dictionary by considering all past information. Specifically, the optimization problem is formulated as follows:

\begin{equation}
	\mathbf{D}^{(t)} = \underset{\mathbf{D} \in \mathrm{C}}{\arg\min} \left\{ \underbrace{\ell\left(\mathbf{D}, \mathbf{x}^{(1)}, \mathbf{y}^{(1)}\right) + \cdots + \ell\left(\mathbf{D}, \mathbf{x}^{(t)}, \mathbf{y}^{(t)}\right)}_{ \mathcal{F}^{(t)}(\mathbf{D})} \right\}.
\end{equation}

In this approach, the  function \(\mathcal{F}^{(t)}(\mathbf{D})\) aggregates all past observations, providing a more accurate update by leveraging the full history of data points \((\mathbf{x}^{(i)}, \mathbf{y}^{(i)})\) for \(i = 1, \dots, t\). This allows for more robust dictionary updates compared to the first-order method, which relies only on the most recent data point. 

In the following section, we present a generalization of these online methods to solve multidimensional problems.

\section{Generalized dictionary learning}\label{sec3}
In this section, we present a generalization of the dictionary learning problem to handle multidimensional data, unlike previous work \cite{zhang}, which is limited to 2D data. 

Consider a set of multidimensional signals, denoted as \(\{ \mathcal{Y}_{i} \}_{i=1}^{n} \subset \mathbb{R}^{M_{1}\times n \times M_{2} \times \dots \times M_{N}}\), and a dictionary \(\mathcal{D} \in \mathbb{R}^{M_{1} \times d \times M_{2} \times \dots \times M_{N}}\) defined as \(\mathcal{D} = [\mathcal{D}_{1}, \dots, \mathcal{D}_{d}]\), where each atom \(\mathcal{D}_{k} \in \mathbb{R}^{M_{1}\times 1 \times M_{2} \times \dots \times M_{N}}\) for \(k = 1, \dots, d\). 

We can express each signal \(\mathcal{Y}_{i}\) as a t-linear combination of the dictionary atoms \(\mathcal{D}_{k}\):
\[
\mathcal{Y}_{i} = \sum_{k=1}^{d} \mathcal{D}_{k} \ast_{t} \mathcal{X}_{k,i} = \mathcal{D} \ast_{t} \mathcal{X}_{i},
\]
where \(\mathcal{X}_{i} = [\mathcal{X}_{1,i}, \dots, \mathcal{X}_{d,i}] \in \mathbb{R}^{d \times 1 \times M_{2} \times \dots \times M_{N}}\) and \(\ast_{t}\) denotes the t-product operation.

The goal is to identify a dictionary that promotes sparsity in the signal representation. This leads us to formulate the multidimensional dictionary learning problem as follows
\[
\min_{\mathcal{D}, \mathcal{X}_{i}} \sum_{i=1}^{n} \left\| \mathcal{D} \ast_{t} \mathcal{X}_{i} - \mathcal{Y}_{i} \right\|_{F}^{2} \quad \text{subject to } \| \mathcal{X}_{i} \|_{TS} \leq T, \quad i = 1, \dots, n.
\]
The tubal sparsity \(\| \cdot \|_{TS}\) is defined as the count of non-zero block tensors \(\mathcal{X}_{i}(k, :, :, \cdots, :)\) in \(\mathcal{X}_{i}\), for \(1 \leq k \leq d\).

Since minimizing \(\| \cdot \|_{TS}\) is NP-hard, we instead use the \(\| \cdot \|_{1}\) norm as a convex relaxation. This norm is defined as:
\[
\| \mathcal{X} \|_{1} = \sum_{i_{1}, i_{2}, \cdots, i_{N}} | \mathcal{X}(i_{1}, i_{2}, \cdots, i_{N}) |.
\]
Thus, the multidimensional dictionary learning problem can be reformulated as:
\begin{equation}\label{DLP1}
	\min_{\mathcal{D}_{k} \in \mathcal{C}, \mathcal{X}_{i}} \sum_{i=1}^{n} \sum_{k=1}^{d} \left\| \mathcal{D}_{k} \ast_{t} \mathcal{X}_{i} - \mathcal{Y}_{i} \right\|_{F}^{2} + \lambda \| \mathcal{X}_{i} \|_{1},
\end{equation}
where \(\mathcal{C} = \{ \mathcal{D} \in \mathbb{R}^{M_{1} \times \dots \times M_{N}} : \| \mathcal{D} \|_{F} = 1 \}\).

Problem \eqref{DLP1} is generally non-convex, though it is convex with respect to each variable \(\mathcal{D}\) and \(\mathcal{X}\) individually. A common solution approach is alternating minimization: first solving the sparse coding problem with the dictionary fixed, then updating the dictionary while keeping the sparse tensor fixed.

\subsection{Generalized Online Dictionary Learning}
Consider the following loss function:
\begin{equation}\label{F}
	\mathcal{F}(\mathcal{D}; \mathcal{Y}) = \min_{\mathcal{X}} \mathcal{L}(\mathcal{X}, \mathcal{D};\mathcal{Y}),
\end{equation}
where
\begin{equation}
	\mathcal{L}(\mathcal{X}, \mathcal{D};\mathcal{Y}) = \frac{1}{2} \Vert \mathcal{D} \ast_{t} \mathcal{X} - \mathcal{Y} \Vert_{F}^{2} + \lambda \Vert \mathcal{X} \Vert_{1}.
\end{equation}
The online multidimensional dictionary learning problem is then formulated as:
\begin{equation}
	\min_{\mathcal{D} \in \mathcal{C}} \frac{1}{N} \sum_{i=1}^{N} \mathcal{F}(\mathcal{D}; \mathcal{Y}_{i}),
\end{equation}
where each $\mathcal{Y}_{i}$ is drawn randomly from the dataset.

As mentioned earlier, this problem is not jointly convex; however, it is convex with respect to each variable individually. Therefore, we adopt an alternating optimization strategy that iterates between sparse coding and dictionary update steps.
In this work, we employ two online methods to solve this problem: a first-order  method  Projected Stochastic Gradient Descent (PSGD), and a second-order method based on the ISTA algorithm.
\subsection{First order method}
\subsubsection{Sparse coding stage}
In the sparse coding stage, we fix the dictionary $\mathcal{D}$ and we consider the following problem:
\begin{equation}\label{SC}
	\min_{ \mathcal{X}} \Vert \mathcal{D}\ast_{t} \mathcal{X}
	-\mathcal{Y}\Vert_{F}^{2}~~ \text{subject to }  \Vert \mathcal{X} \Vert_{0} \leq T.
\end{equation}

To solve the problem, we utilize the orthogonal matching pursuit (OMP)\cite{Elad}\cite{addi} algorithm, a greedy approach that addresses the norm $\Vert . \Vert_{0}$ optimization problem. This algorithm is known for its simplicity and effectiveness. The idea behind OMP is to identify the dictionary atoms that are most correlated with the signal and solve a simple least squares problem based on those atoms. The OMP algorithm is illustrated in algorithm \ref{OMP}.
\begin{algorithm}
	\Large
	\caption{Orthogonal matching pursuit OMP}
	\begin{algorithmic}
		\large
		\Require 
		dictionary $\mathcal{D}$, signal $\mathcal{Y}$, maximum of nonzero elements $K$, tolerance $\epsilon$.
		\Ensure $\mathcal{D}\ast_{t}\mathcal{X}=\mathcal{Y}$, with $\Vert \mathcal{X}\Vert_{0} \leq K $
		\State 1: $\mathcal{I}=[\emptyset]$, $~ ~ \mathcal{R}=\mathcal{Y}, k=1;$
		\While{$ |\mathcal{I}| \leq K_{max} ~ ~ and ~ ~ \Vert \mathcal{R} \Vert_{F} > \varepsilon $}
		\State  3: $[i]=argmax_{[i]}\Vert (\mathcal{D}^{T}(i,:,\cdots,:)\ast_{t}\mathcal{R}) \Vert_{F}; $
		\State 4: $\mathcal{I}=[\mathcal{I},i];$ 
		\State 5: $ \mathcal{A}=argmin_{\mathcal{U}} \Vert \mathcal{Y} -\mathcal{D}(:,\mathcal{I},\cdots,:)\ast_{t}\mathcal{U} \Vert_{F}^{2};$
		\State 6: $ \mathcal{R}=\mathcal{Y}-\mathcal{D}\ast_{t}\mathcal{X}$, with $\mathcal{X}(\mathcal{I},\cdots,:)=\mathcal{A};$
		\State 7: $ k=k+1; $
		\EndWhile
	\end{algorithmic}
	\label{OMP}
\end{algorithm}

The least square problem in Step 5, can be solved using Cholesky factorization. We generalize the method given in \cite{Elad}.
Consider a signal \(\mathcal{Y}\) that has a sparse representation in a dictionary \(\mathcal{D}\):
\begin{align*}
	\mathcal{Y}&=\sum_{k=1}^{K}\mathcal{D}(:,i_{k},:,\dots,:)\ast_{t}\mathcal{X}(i_{k},1,:,\dots,:)\\
	&=\mathcal{D}(:,\mathcal{I},\dots,:)\ast_{t}\mathcal{X}(\mathcal{I},1,\dots,:),
\end{align*} 
such that $\mathcal{I}=[i_{1},\dots,i_{K}]$.
We multiply by $\mathcal{D}^{T}(\mathcal{I},\cdots,:)$, then the problem turns to
\begin{equation}
	\mathcal{D}^{T}(\mathcal{I},\cdots,:)\ast_{t}\mathcal{Y}=(\mathcal{D}^{T}(\mathcal{I},\cdots,:)\ast_{t}\mathcal{D}(:,\mathcal{I},\cdots,:))\ast_{t}\mathcal{X}(\mathcal{I},1,\cdots,:).
\end{equation} 
The tensor  $\mathcal{D}^{T}(\mathcal{I},\cdots,:)\ast_{t}\mathcal{D}(:,\mathcal{I},\cdots,:)$ is symmetric, and positive definite due to the orthogonalization process which ensures the selection of linearly independent atoms.
Then we use the Choleseky factorization to solve the problem.
In each iteration $k$  of OMP we add new block  $D(:,i_{k},:,:)$ to $D(:,\mathcal{I},\cdots,:)$. 
Therefore  the Cholesky factorization of the tensor  $\mathcal{D}^{T}(\mathcal{I},\cdots,:)\ast_{t}\mathcal{D}(:,\mathcal{I},\cdots,:)$ requires only the computation of its last tensor row and tensor column as demonstrated in Theorem  \ref{Cholesky}. 
Then using the Cholesky factorization we can write the equation in Step 5 as
\[\mathcal{L}\ast_{t}\mathcal{L}^{T}\ast_{t}\mathcal{U}=\mathcal{D}^{T}(:,\mathcal{I},\dots,:)\ast_{t}\mathcal{Y},\]
then the solution to Step 5 is given by
\[\mathcal{U}=\mathcal{L}^{-T}\ast_{t}\mathcal{L}^{-1}\ast_{t}(\mathcal{D}^{T}(:,\mathcal{I},\dots,:)\ast_{t}\mathcal{Y}).\]
The following lemma and theorem provide fundamental tools for tensor analysis within the t-product framework. The lemma establishes the existence of tensor square roots for specific tensor structures, while the theorem builds upon this result to derive the Cholesky factorization of block tensors, a crucial step in the OMP algorithm.
\begin{lemma}\label{lem1}
	Let \(\mathcal{B} \in \mathbb{R}^{1 \times 1 \times I_{3} \times \dots \times I_{N}}\) be a tensor with  positive entries. Then, the square root of \(\mathcal{B}\) under the t-product exists and is given by \(\sqrt{\mathcal{B}}\), whose frontal slices are defined as
	\[
	\left(\sqrt{\mathcal{B}}\right)^{(i)} = \sqrt{\mathcal{B}^{(i)}}, \quad \text{for } i = 1, \dots, P,
	\]
	where \(P = I_{3} \cdot I_{4} \cdot \dots \cdot I_{N}\).
\end{lemma}

\begin{proof}
	Since all entries of \(\mathcal{B}\) are positive, each frontal slice \(\mathcal{B}^{(i)}\) is a positive scalar for \(i = 1, \dots, P\).
	Define the tensor \(\mathcal{C} \in \mathbb{R}^{1 \times 1 \times I_{3} \times \dots \times I_{N}}\) by
	\[
	\mathcal{C}^{(i)} = \sqrt{\mathcal{B}^{(i)}} \quad \text{for } i = 1, \dots, P.
	\]
	It is easy to see that

	\[
	\mathcal{C} *_{t} \mathcal{C} = \mathcal{B}.
	\]
	Thus, \(\mathcal{C}\) is the square root of \(\mathcal{B}\) under the t-product. Therefore,
	\[
	\sqrt{\mathcal{B}} = \mathcal{C}.
	\]
\end{proof}

\begin{theorem}\label{Cholesky}
Consider the symmetric positive definite tensor $\mathcal{A}\in \mathbb{R}^{I_{1}\times I_{1}\times I_{3}\times\dots\times I_{N}} $  such that:
\[ \mathcal{A}=\left[\begin{array}{cc}
	\tilde{\mathcal{A}} & \mathcal{V} \\
	\mathcal{V}^{T} & \mathcal{C}
\end{array}\right],\]
with $\tilde{\mathcal{A}}\in \mathbb{R}^{(I_{1}-1)\times (I_{1}-1)\times I_{3}\times\dots\times I_{N}} $ is symmetric positive definite,  $\mathcal{V}\in \mathbb{R}^{(I_{1}-1)\times 1\times I_{3}\times\dots\times I_{N}}$ a block tensor and $\mathcal{C}\in \mathbb{R}^{1\times 1\times I_{3}\times\dots\times I_{N}}$.
 The Cholesky factorization of $\mathcal{A}$ is given by
\begin{equation}\label{cholesky}
	\mathcal{L}=\left[\begin{array}{cc}
		\tilde{\mathcal{L}} & 0 \\
		\mathcal{W}^{T} & \sqrt{\mathcal{C}-\mathcal{W}^{T}\ast_{t}\mathcal{W}}
	\end{array}\right],
\end{equation}
Where $\tilde{\mathcal{L}}$ is the lower triangular tensor from the Cholesky factorization of  $\tilde{\mathcal{A}}$, and the tensor $\mathcal{W}$ is the solution to:
\[ \tilde{\mathcal{L}}\ast_{t}\mathcal{W}=\mathcal{V}.
\]
\end{theorem}
\begin{proof}
Let \( P = I_3 \times \dots \times I_N \). Consider the Fourier transform of the tensor \( \mathcal{A} \), denoted by \( \hat{\mathcal{A}} \). The \( i \)-th frontal slice of \( \hat{\mathcal{A}} \) is \( \hat{\mathcal{A}}^{(i)} \) for \( i = 1, \dots, P \).

Since \( \mathcal{A} \) is symmetric positive definite, there exists a lower triangular tensor \( \mathcal{L} \) such that \cite{REICHEL2022}:
\[ \mathcal{A} = \mathcal{L} \ast_t \mathcal{L}^T. \]
In the Fourier domain, this becomes:
\[ \hat{\mathcal{A}}^{(i)} = \hat{\mathcal{L}}^{(i)} \hat{\mathcal{L}}^{(i)T}, \quad \text{for } i = 1, \dots, P.\]

with \( \hat{\mathcal{A}}^{(i)} \):
\[ \hat{\mathcal{A}}^{(i)} = \begin{pmatrix}
	\hat{\tilde{\mathcal{A}}}^{(i)} & \hat{\mathcal{V}}^{(i)} \\
	\hat{\mathcal{V}}^{(i)T} & \hat{\mathcal{C}}^{(i)}
\end{pmatrix}, \]
where \( \hat{\tilde{\mathcal{A}}}^{(i)} \), \( \hat{\mathcal{V}}^{(i)} \), and \( \hat{\mathcal{C}}^{(i)} \) are the Fourier transforms of \( \tilde{\mathcal{A}} \), \( \mathcal{V} \), and \( \mathcal{C} \) at frequency \( i \).

Assuming the Cholesky factorization of \( \hat{\tilde{\mathcal{A}}}^{(i)} \) is \( \hat{\tilde{\mathcal{L}}}^{(i)} \hat{\tilde{\mathcal{L}}}^{(i)T} \), we solve for \( \hat{\mathcal{W}}^{(i)} \):
\[
\hat{\tilde{\mathcal{L}}}^{(i)} \hat{\mathcal{W}}^{(i)} = \hat{\mathcal{V}}^{(i)}.
\]

Since \( \hat{\mathcal{A}}^{(i)} \) is positive definite, the Schur complement
\[
\hat{\mathcal{C}}^{(i)} - \hat{\mathcal{V}}^{(i)T} (\hat{\tilde{\mathcal{A}}}^{(i)})^{-1} \hat{\mathcal{V}}^{(i)} > 0.
\]
Because \( \hat{\mathcal{W}}^{(i)} = (\hat{\tilde{\mathcal{L}}}^{(i)})^{-1} \hat{\mathcal{V}}^{(i)} \), we have:
\[
\hat{\mathcal{W}}^{(i)T} \hat{\mathcal{W}}^{(i)} = \hat{\mathcal{V}}^{(i)T} (\hat{\tilde{\mathcal{A}}}^{(i)})^{-1} \hat{\mathcal{V}}^{(i)}.
\]
Then 
\[
\hat{l}_{22}^{(i)} = \sqrt{\hat{\mathcal{C}}^{(i)} - \hat{\mathcal{W}}^{(i)T} \hat{\mathcal{W}}^{(i)}}
.\]
Then  \( \hat{\mathcal{L}}^{(i)} \) is given by \cite{Elad}:
\[
\hat{\mathcal{L}}^{(i)} = \begin{pmatrix}
	\hat{\tilde{\mathcal{L}}}^{(i)} & 0 \\
	\hat{\mathcal{W}}^{(i)T} & \hat{l}_{22}^{(i)}
\end{pmatrix}.
\]

Thus, the Cholesky factorization of \( \mathcal{A} \) is:
\[
\mathcal{A} = \mathcal{L} \ast_t \mathcal{L}^T,
\]
with:
\[
\mathcal{L} = \begin{pmatrix}
	\tilde{\mathcal{L}} & 0 \\
	\mathcal{W}^T & \sqrt{\mathcal{C} - \mathcal{W}^T \ast_t \mathcal{W}}
\end{pmatrix},
\]
where \( \tilde{\mathcal{L}} \) and \( \mathcal{W} \) satisfy:
\[
\tilde{\mathcal{L}} \ast_t \mathcal{W} = \mathcal{V}.
\]
The tensor $\mathcal{C} - \mathcal{W}^T \ast_t \mathcal{W}$ is positive due to its coefficient are positive, which implies that the square tensor $\sqrt{\mathcal{C} - \mathcal{W}^T \ast_t \mathcal{W}}$ exists from  Lemma \ref{lem1}.
\end{proof}

\subsection{Dictionary update}
Projected stochastic gradient descent (PSGD) is used to implement the dictionary update step:
\begin{equation}
\mathcal{D}^{(t)}=\arg\min\limits_{\mathcal{D}}\sum_{k=1}^{t}L(\mathcal{X}^{(k)},\mathcal{D}).
\end{equation}
The general step of PSGD for the multidimensional problem is given by:
\begin{equation}
	\mathcal{D}^{(t)}=Proj_{\mathcal{C}}(\mathcal{D}^{(t-1)}-\eta^{(t)}\nabla \mathcal{F}(\mathcal{D}^{(t-1)};\mathcal{Y}^{(t)})),
\end{equation}
Such that $\eta^{(t)}$ is the step size or the learning rate, in \cite{liu} they use a diminishing learning rate $\eta^{(t)}=\dfrac{a}{b+t}$, where $a$ and $b$ have to be well chosen in a data set-dependent way.
In the following proposition, we present the gradient of the function $\mathcal{F}$.
\begin{proposition}\label{Gradient}
	Consider the function $\mathcal{F}$ as defined in \eqref{F}.$\mathcal{F}$ is differentiable and the gradient is given by 
	\[ \nabla \mathcal{F}(\mathcal{D})=
	(\mathcal{D}\ast_{t}\mathcal{X}-\mathcal{Y})\ast_{t}\mathcal{X}^{T}.\]
	
\end{proposition} 
\begin{proof}
	The $t$-product is a linear operation defined in normed space, then $\mathcal{F}$ is a composition of differentiable applications then it is differentiable. Let's compute the gradient of this function:
	\begin{align*}
		\mathcal{F}(\mathcal{D}+\mathcal{H})-\mathcal{F}(\mathcal{D})&=\dfrac{1}{2}\Vert \mathcal{D}\ast_{t}\mathcal{X}+\mathcal{H}\ast_{t}\mathcal{X}-\mathcal{Y}\Vert_{F}^{2}+ \lambda\Vert \mathcal{X}\Vert_{1}-\dfrac{1}{2}\Vert\mathcal{D}\ast_{t}\mathcal{X}-\mathcal{Y}\Vert_{F}^{2}-\lambda\Vert \mathcal{X}\Vert_{1}\\
		&=<\mathcal{H}\ast_{t}\mathcal{X},\mathcal{D}\ast_{t}\mathcal{X}-\mathcal{Y}>+\dfrac{1}{2}\Vert \mathcal{H}\ast_{t}\mathcal{X}\Vert_{F}^{2}\\
		&=<\mathcal{H},(\mathcal{D}\ast_{t}\mathcal{X}-\mathcal{Y})\ast_{t}\mathcal{X}^{T}>+\epsilon(\Vert \mathcal{H}\Vert_{F}).
	\end{align*}
	Then the gradient is given by:
	\[ \nabla \mathcal{F}(\mathcal{D})=
	(\mathcal{D}\ast_{t}\mathcal{X}-\mathcal{Y})\ast_{t}\mathcal{X}^{T}.	\]
\end{proof}

\begin{algorithm}
	\Large
	\caption{Online projected stochastic gradient descent (OPSGD)}
	\begin{algorithmic}
		\large
		\State \textbf{Initialize:} $\mathcal{D}^{(0)}$ with a random dictionary
		\For{t=1:T}
		\State sample a signal $\mathcal{Y}^{(t)}$.
		\State  solve the sparse coding problem using Algorithm \ref{OMP}.
		\State update the dictionary using:\\
		$\mathcal{D}^{(t)}=Proj_{C}(\mathcal{D}^{(t-1)}-\eta^{(t)}(\mathcal{D}^{(t)}\ast_{t}\mathcal{X}-\mathcal{Y})\ast_{t}\mathcal{X}^{T}).$
		\EndFor
		
	\end{algorithmic}
	\label{ODL}
\end{algorithm}

\subsection{Second order method}
Let's consider the problem of the dictionary update
\begin{equation}
	\mathcal{D}^{(t)}=\arg\min\limits_{\mathcal{D}}\dfrac{1}{t}\sum_{k=1}^{t}L(\mathcal{X}^{(k)},\mathcal{D}),
\end{equation}
such that $\mathcal{X}^{(k)}$ is already given from Algorithm \ref{OMP}.\\
 The function 
$\bar{\mathcal{F}}^{(t)}(\mathcal{D})=\dfrac{1}{t}\sum_{k=1}^{t}L(\mathcal{X}^{(k)},\mathcal{D}^{(k)})$ aggregates the past information from previous iterations.
Then we update the dictionary as:
\begin{equation}
\mathcal{D}^{(t)}=\underset{\mathcal{D}}{\arg\min}  \bar{\mathcal{F}}^{(t)}(\mathcal{D})+ Proj_{\mathcal{C}}(\mathcal{D}).
\end{equation}
This problem can be solved efficiently using the proximal gradient or ISTA \cite{beck}.
First, we compute the gradient of the function $\mathcal{F}^{(t)}$, which can be concluded from Proposition \ref{Gradient}.
\begin{equation}
\nabla\bar{\mathcal{F}}^{(t)}=\dfrac{1}{t}\sum_{k=1}^{t}(\mathcal{D}\ast_{t}\mathcal{X}^{(k)}-\mathcal{Y}^{(k)})\ast_{t}(\mathcal{X}^{(k)})^{T}.
\end{equation}
We define $L$ as the Lipschitz constant of the function $\bar{\mathcal{F}}^{(t)}$.
\begin{algorithm}
	\Large
	\caption{Online second order method }
	\begin{algorithmic}
		\large
		\State \textbf{Initialize:} $\mathcal{D}^{(0)}$ with a random dictionary
		\For{t=1:T}
		\State sample a signal $\mathcal{Y}^{(t)}$.
		\State solve the sparse coding problem using Algorithm \ref{OMP}.
		\State $\mathcal{B}^{(t)}=\mathcal{B}^{(t-1)}+\mathcal{Y}^{(t)}\ast_{t}(\mathcal{X}^{(t)})^{T}$
		\State $\mathcal{A}^{(t)}=\mathcal{A}^{(t-1)}+\mathcal{X}^{(t)}\ast_{t}(\mathcal{X}^{(t)})^{T}$
		\State update the dictionary using:\\
		$\mathcal{D}^{(t)}=Proj_{C}(\mathcal{D}^{(t-1)}-{\frac{1}{tL}}(\mathcal{D}^{(t-1)}\ast_{t}\mathcal{A}^{(t)}-\mathcal{B}^{(t)})).$
		\EndFor
		
	\end{algorithmic}
	\label{ODL2}
\end{algorithm}
\section{Multidimensional completion problem}\label{MCO}
Completion refers to filling in or estimating the missing values in a given data. Consider a given data $\mathcal{Y}$ with missing samples and $\Omega$ the set of available samples, and let $\mathcal{W}$ be a binary mask such that:

\[\mathcal{W}_{i_{1}i_{2}\dots i_{N}}=\left\{\begin{array}{ccc}
	1 & \text{if}~~ (i_{1},i_{2},\dots ,i_{N})\in \Omega \\
	0 & 	\text{else} \end{array}\right.
\]
Then we define the completion problem as follows:
\begin{equation}\label{COP}
	\min_{\mathcal{X}} \dfrac{1}{2}\Vert  \mathcal{W}\odot(\mathcal{D}\ast_{t}\mathcal{X}-\mathcal{Y})\Vert_{F} + \lambda\Vert \mathcal{X} \Vert_{1},
\end{equation}
such that $\odot$ is the pointwise product.
The following section presents the method to solve this problem.
                              
\subsection{ Iterative Shrinkage Thresholding with Anderson acceleration}
The Iterative Shrinkage Thresholding Algorithm (ISTA) is a class of first-order methods that can be seen as an extension of classical gradient methods, used to solve linear inverse problems such as \eqref{SC}. This algorithm is known for its simplicity and efficiency in addressing this type of problem. However, ISTA is known to converge slowly and exhibits a sublinear global rate of convergence.
In \cite{beck}, an extrapolation method is used to accelerate the algorithm, resulting in the Fast Iterative Shrinkage-Thresholding Algorithm (FISTA). In this work, we employ a different acceleration method, known as Anderson Acceleration (AA), to improve the performance of ISTA.
The Anderson Acceleration (AA) method for fixed-point iterations, originally developed by D. G. Anderson, was primarily used in electronic structure computations. Despite its effectiveness, this method has not been as widely adopted as other acceleration techniques. In our specific problem, however, it demonstrates superior results in terms of precision compared to FISTA and other acceleration methods, as will be shown in the numerical results section. The new accelerated algorithm is presented in Algorithm \ref{AISTA} and is based on the work in \cite{walker}, \cite{JUNZI}, and \cite{brezinski}.

Let's turn to  the problem \eqref{SC}:
\begin{align}\label{SCFISTA}
	\min_{\mathcal{X}}L(\mathcal{X},\mathcal{D}^{(t)};\mathcal{Y})=& \frac{1}{2}\Vert \mathcal{D}^{(t)}\ast_{t} \mathcal{X}
	-\mathcal{Y}\Vert_{F}^{2} + \lambda\Vert \mathcal{X} \Vert_{1}, \\
	=& G(\mathcal{X}) + \lambda\Vert \mathcal{X}\Vert_{1}
\end{align}

The general solution to the problem \eqref{SCFISTA} is given by
\begin{equation}\label{Thresh}
	\mathcal{X}_{k}=\mathcal{T}_{\lambda t}(\mathcal{X}_{k-1}-t\mathcal{D}^{T}\ast_{t}(\mathcal{D}\ast_{t}\mathcal{X}_{k-1}-\mathcal{Y}), 
\end{equation}
such that  $\mathcal{T}_{\alpha}: \mathbb{R}^{M_{1}\times M_{2}\times\dots\times M_{N}}\rightarrow \mathbb{R}^{M_{1}\times M_{2}\times\dots\times M_{N}} $ is the shrinkage operator defined by:
\begin{equation}
	T_{\alpha}(\mathcal{X})_{i_{1}i_{2}\dots i_{N}}=(\vert \mathcal{X}_{i_{1}i_{2}\dots i_{N}}\vert -\alpha)_{+}sgn(\mathcal{X}_{{i_{1}i_{2}\dots i_{N}}}).
\end{equation}
The parameter $t=\dfrac{1}{L}$,
such that $L$ is the Lipschitz continuous gradient of the function $G$.

\begin{proposition}
	Consider the function
	$G(\mathcal{X})=\frac{1}{2}\Vert \mathcal{D}\ast_{t}\mathcal{X}-\mathcal{B} \Vert_{F}^{2}.$
	$G$ is  differentiable, and the gradient is given by:
	\[ \nabla G(\mathcal{X})=\mathcal{D}^{T}\ast_{t}(\mathcal{D}\ast_{t}\mathcal{X}-\mathcal{B}).\] 
\end{proposition}
\begin{proof}
	\begin{align}
		G(\mathcal{X}+\mathcal{H})-G(\mathcal{X})&=\dfrac{1}{2}\Vert \mathcal{D}\ast_{t}\mathcal{X}+\mathcal{D}\ast_{t}\mathcal{H}-\mathcal{B}\Vert_{F}^{2}-\dfrac{1}{2}\Vert\mathcal{D}\ast_{t}\mathcal{X}-\mathcal{B}\Vert_{F}^{2}\\
		&=<\mathcal{D}\ast_{t}\mathcal{H},\mathcal{D}\ast_{t}\mathcal{X}-\mathcal{B}>+\dfrac{1}{2}\Vert \mathcal{D}\ast_{t}\mathcal{H}\Vert_{F}^2 \\
		&=<\mathcal{H},\mathcal{D}^{T}\ast_{t}(\mathcal{D}\ast_{t}\mathcal{X}-\mathcal{B})>+ \epsilon(\mathcal{H})\Vert \mathcal{H}\Vert_{F}
	\end{align}
	then the gradient is given by:
	\begin{equation}
		\nabla G(\mathcal{X})=\mathcal{D}^{T}\ast_{t}(\mathcal{D}\ast_{t}\mathcal{X}-\mathcal{B}).
	\end{equation}
\end{proof}
\begin{proposition}
	Lipschitz continuous gradient $L(G)$ of the function $G$ is given by:
	\[ L(G)=\rho(\mathcal{D}^{T} \ast_{t} \mathcal{D}) =  \max_{\lambda \in \Lambda(\mathcal{D}^{T} \ast_{t} \mathcal{D})}\Vert(\lambda) \Vert_{F}. \]
\end{proposition}
\begin{proof}
	From \cite{eigentube}, the spectral radius is given by 
	\[
	\rho(\mathcal{D}^{T} \ast_{t} \mathcal{D}) =  \max_{\lambda \in \Lambda(\mathcal{D}^{T} \ast_{t} \mathcal{D})}\Vert(\lambda) \Vert_{F}.
	\]
	Although the referenced work addresses only the case of third-order tensors, this property can be generalized to tensors of arbitrary order \(N\).
	We then have
	\begin{align}
		\Vert \nabla G(\mathcal{X}) - \nabla G(\mathcal{Z}) \Vert_{F} &= \Vert \mathcal{D}^{T} \ast_{t} (\mathcal{D} \ast_{t} \mathcal{X} - \mathcal{Y}) - \mathcal{D}^{T} \ast_{t} (\mathcal{D} \ast_{t} \mathcal{Z} - \mathcal{Y}) \Vert_{F} \\
		&= \Vert \mathcal{D}^{T} \ast_{t} \mathcal{D} \ast_{t} (\mathcal{X} - \mathcal{Z}) \Vert_{F} \\
		&\leq \rho(\mathcal{D}^{T} \ast_{t} \mathcal{D}) \Vert \mathcal{X} - \mathcal{Z} \Vert_{F},
	\end{align}
	where the inequality holds because the spectral radius \(\rho(\mathcal{D}^{T} \ast_{t} \mathcal{D})\) is the largest eigenvalue of \(\mathcal{D}^{T} \ast_{t} \mathcal{D}\).
	Thus, the Lipschitz constant for the gradient is given by
	\[
	L(G) = \rho(\mathcal{D}^{T} \ast_{t} \mathcal{D}).
	\]
\end{proof}

Let's consider the fixed point iteration in \eqref{Thresh} such that:
\begin{equation}
	\mathcal{G}(\mathcal{X})=\mathcal{T}_{\frac{\lambda }{L}}(\mathcal{X}-\frac{1}{L}\mathcal{D}^{T}\ast_{t}(\mathcal{D}\ast_{t}\mathcal{X}-\mathcal{Y}).
\end{equation}
AA consist of choosing $\mathcal{X}_{0}$, and $m \geq 1$, we set $\mathcal{G}(\mathcal{X}_{0})=\mathcal{X}_{1}$. For $k=1\dots$, 
we set $m_{k}=min\{m,k\}$, and $\mathcal{F}_{k}=(\nabla F_{k-m_{k}},\dots,\nabla F_{k-1})$, such that $\nabla F_{i}=F_{i+1}-F_{i}$ and $F_{i}=\mathcal{G}(\mathcal{X}_{i})-\mathcal{X}_{i}$. After we determine the tensor $\mathcal{U}$  that solves the least squares problem \eqref{23} using the QR decomposition generalized to p-order tensor presented in Algorithm \ref{T-QR}.
\begin{equation}\label{23}
	\min_{\mathcal{U}}\Vert F_{k}- \mathcal{F}_{k}\ast_{t}\mathcal{U} \Vert_{F}.
\end{equation}
Then the next iteration of the algorithm is
\begin{equation}
	\mathcal{X}_{k+1}=\mathcal{X}_{k}+F_{k}-(\bar{\mathcal{Z}_{k}}+\mathcal{F}_{k})\ast_{t}\mathcal{U},
\end{equation}
such that $\bar{\mathcal{Z}_{k}}=(\nabla \mathcal{X}_{k-m_{k}},\dots,\nabla \mathcal{X}_{k-1})$, and $\nabla \mathcal{X}_{i}=\mathcal{X}_{i+1}-\mathcal{X}_{i}$.
The complete accelerated ISTA-AA algorithm is illustrated in Algorithm \ref{AISTA}.
\begin{algorithm}[H]
	\caption{ ISTA-AA}
	\begin{algorithmic}[1] 
		
		\Require $L=2\lambda_{max}(\mathcal{D}^{T}\ast_{t}\mathcal{D})$, $t=\frac{1}{L}$, $m\geq 1$.
		\State \textbf{Step 0 \\}
		$\mathcal{X}_{0}\in \mathbb{R}^{M_{1}\times M_{2}\times M_{3}}$, $\mathcal{G}_{0}=\mathcal{X}_{1}-\mathcal{X}_{0}$ 
		\State \textbf{Step k  repeat until convergence}
		\State Set $m_{k}=\min\{m,k\}$;
		
		\State Solve  $~~~~\min\limits_{\mathcal{U}}\Vert F_{k}-\mathcal{F}_{k}\ast_{t}\mathcal{U}\Vert_{F}$;
		\State  $\mathcal{X}_{k+1}=\mathcal{X}_{k}+F_{k}-(\mathcal{Z}_{k}+\mathcal{F}_{k})\ast_{t}\mathcal{U}$; 
	\end{algorithmic}
	\label{AISTA}
\end{algorithm}     

\begin{algorithm}
	\Large
	\caption{T-QR }
	\begin{algorithmic}
		\large
		\Require
		A tensor $\mathcal{A}\in \mathbb{R}^{I_{1}\times I_{2}\times \dots \times I_{N}}$
		\Ensure
		$\mathcal{A}=\mathcal{Q}\ast_{t}\mathcal{R}$;
		\State 1: $\rho=I_{3}I_{4}\dots I_{N}$
		\For{$i=1:\rho$}
		\State  $\mathcal{A}=fft(\mathcal{A},[],i)$;
		\EndFor
		\For{$i=1:\rho$} 
		\State $[Q,R]=qr(\mathcal{A}(:,:,i));$ 
		\State  $\mathcal{Q}(:,:,i)=Q$, $\mathcal{R}(:,:i)=R$, 
		\EndFor
		\For{ $i=p,\dots 3$}
		\State	$\mathcal{Q}=ifft(\mathcal{Q},[],i)$;
		$\mathcal{R}=ifft(\mathcal{R},[],i)$;
		\EndFor
		
	\end{algorithmic}
	\label{T-QR}
\end{algorithm}
Let's turn to the problem  \eqref{COP}, we solve this problem using Algorithm \ref{AISTA}, with a slight modification to the thresholding operator:
\begin{equation}
	\mathcal{G}(\mathcal{X}) = \mathcal{T}_{\lambda t}(\mathcal{X} - t \mathcal{D}^{T} \ast_{t} (\mathcal{W} \odot (\mathcal{D} \ast_{t} \mathcal{X} - \mathcal{Y}))).
\end{equation}
This modification ensures that unavailable pixels are ignored during the optimization process.
 
 \subsection{Sufficient Condition for Signal Recovery}
 
 A fundamental concept in sparse recovery theory is the null space property, which provides a sufficient condition for the recovery of sparse signals. In this section, we extend this property to the context of multidimensional tensors. Furthermore, we present both the necessary and sufficient conditions for sparse signal recovery in the tensor framework, as discussed in \cite{Foucart}.  We denote $\mathcal{V}_{S} \in \mathbb{R}^{I_{2} \times 1 \times \dots \times I_{N}}$ as the restriction of $\mathcal{V}$ to indices in $S \subset [I_{2}]$ along the first dimension. We say that $\mathcal{V} \in \mathbb{R}^{I_{2} \times 1 \times \dots \times I_{N}}$ is supported on $S$ if
 \[
 \mathcal{V}_{S} = \mathcal{V}.
 \]
 Let's define the operator associated with the problem \eqref{COP}
\begin{align*}\label{P}
	\mathcal{P}:& \, \mathbb{R}^{I_{2} \times 1 \times I_{3}\times \dots \times I_{N}} \longrightarrow \mathbb{R}^{I_{1} \times 1 \times I_{3} \times \dots \times I_{N}}, \\
	& \mathcal{X} \mapsto \mathcal{W} \odot (\mathcal{D} \ast_{t} \mathcal{X}),
\end{align*}
 $\mathcal{P}$ is a linear operator.
In our problem \eqref{COP} the measurement tensor is defined by $\mathcal{P}$. Then we establish the null space property for the operator $\mathcal{P}$.
\begin{definition}
	Consider a tensor \(\mathcal{A} \in \mathbb{R}^{I_{1} \times I_{2} \times \dots \times I_{N}}\). The null space of \(\mathcal{A}\) under the t-product operation is defined as
	\[
	\text{Null}(\mathcal{A}) = \left\{ \mathcal{V} \in \mathbb{R}^{I_{2} \times 1 \times I_{3} \times \dots \times I_{N}} \ \bigg| \ \mathcal{A} \ast_{t} \mathcal{V} = \mathbf{0}_{\mathbb{R}^{I_{1} \times 1 \times I_{3} \times \dots \times I_{N}}} \right\}.
	\]
\end{definition}  
 \begin{definition}
 	The operator $\mathcal{P} $ is said to satisfy the null space property relative to a set $S \subset [I_{2}]$ if
 	\begin{equation}
 		\Vert \mathcal{V}_{S} \Vert_{1} < \Vert \mathcal{V}_{\bar{S}} \Vert_{1} \quad \text{for all} \quad \mathcal{V} \in Null(\mathcal{P})  \setminus \{0\},
 	\end{equation}
 	where $\bar{S}$ denotes the complement of $S$ in $[I_{2}]$. Additionally, $\mathcal{P}$ is said to satisfy the null space property of order $s$ if it satisfies the null space property relative to any set $S \subset [I_{2}]$ with $|S| \leq s$.
 \end{definition}
 
 \begin{theorem}
 	Let $\mathcal{P}$ be the operator defined above. Then, for every tensor $\mathcal{X} \in \mathbb{R}^{I_{2} \times 1 \times \dots \times I_{N}}$ with support $S$ in the first dimension, $\mathcal{X}$ is the unique solution to problem \eqref{SCFISTA} with $\mathcal{Y} = \mathcal{P}(\mathcal{X})$ if and only if $\mathcal{P}$ satisfies the null space property relative to $S$.
 \end{theorem}
 
 \begin{proof}
 	Assume that for a fixed index set $S$, every tensor $\mathcal{X} \in \mathbb{R}^{I_{2} \times 1 \times \dots \times I_{N}}$ supported on $S$ is the unique minimizer of $\Vert \mathcal{Z} \Vert_{1}$ subject to $\mathcal{P}(\mathcal{Z}) = \mathcal{P}( \mathcal{X})$. Then, for any $\mathcal{V} \in \ker \mathcal{P} \setminus \{0\}$, the tensor $\mathcal{V}_{S}$ must be the unique minimizer of $\Vert \mathcal{Z} \Vert_{1}$ subject to $\mathcal{P}(\mathcal{Z})  = \mathcal{P} (\mathcal{V}_{S})$. However, we observe that
 	\[
 	\mathcal{P}(-\mathcal{V}_{\bar{S}}) = \mathcal{P}(\mathcal{V}_{S}) ,
 	\]
 	and $-\mathcal{V}_{\bar{S}} \neq \mathcal{V}_{S}$ because
 	\[
 	\mathcal{P}(-\mathcal{V}_{\bar{S}} + \mathcal{V}_{S}) = 0
 	\]
 	and $\mathcal{V} \neq 0$. This implies that
 	\[
 	\Vert \mathcal{V}_{S} \Vert_{1} < \Vert \mathcal{V}_{\bar{S}} \Vert_{1},
 	\]
 thereby establishing the null space property relative to $S$.
 	
 	Conversely, assume that the null space property relative to $S$ holds. Let $\mathcal{X}$ be a tensor supported on $S$, and let $\mathcal{Z}$ be another tensor such that $\mathcal{P} (\mathcal{Z}) = \mathcal{P} (\mathcal{X})$. Define $\mathcal{V} = \mathcal{X} - \mathcal{Z} \in Null(\mathcal{P}) \setminus \{0\}$. By the null space property, we have
 	\begin{align*}
 		\Vert \mathcal{X} \Vert_{1} &\leq \Vert \mathcal{X} - \mathcal{Z} \Vert_{1} + \Vert \mathcal{Z} \Vert_{1} \\
 		&\leq \Vert \mathcal{V}_{S} \Vert_{1} + \Vert \mathcal{Z} \Vert_{1} \\
 		&< \Vert \mathcal{V}_{\bar{S}} \Vert_{1} + \Vert \mathcal{Z}_{S} \Vert_{1} \\
 		&< \Vert -\mathcal{Z}_{\bar{S}} \Vert_{1} + \Vert \mathcal{Z}_{S} \Vert_{1} \\
 		&< \Vert \mathcal{Z} \Vert_{1},
 	\end{align*}
 	which contradicts the assumption that $\mathcal{X}$ is not the unique minimizer unless $\mathcal{X} = \mathcal{Z}$. Therefore, $\Vert \mathcal{X} \Vert_{1}$ is minimal, establishing the uniqueness of the solution.
 \end{proof}

\section{Numerical results}\label{NumR}
This section presents several numerical results to demonstrate the efficiency of the proposed algorithms in completion and denoising applications. All computations were carried out using MATLAB R2023a on a computer with an Intel(R) Core(TM) i7-7600U CPU @ 2.80GHz and 16 GB of RAM.

\subsection{Application to multidimensional completion problem}
In this experiment, we evaluate the efficiency of the method described in Section \ref{MCO} for the image completion problem. We generate random binary masks that conceal 80\% of the pixels in three images (Sweet Peppers, Girl, and Fruits). The images are then reconstructed using the method from Section \ref{MCO}. We compare our approach with two other methods, TNN \cite{KILMER2011} and RTC \cite{ijcai2018}, which are based on the nuclear norm minimization approach.
To evaluate the effectiveness of our method, we use the root mean square error (RMSE) and the peak signal-to-noise ratio (PSNR), defined as:

\[
RMSE = \frac{\Vert \mathcal{Y}_{rec} - \mathcal{Y}_{ori} \Vert_{F}}{\sqrt{I_{1} I_{2} I_{3}}}, \quad
PSNR = 10 \log_{10} \left(\frac{\max(\mathcal{Y}_{ori})}{RMSE}\right)^{2},
\]
where $\mathcal{Y}_{ori} \in \mathbb{R}^{I_{1} \times I_{2} \times I_{3}}$ represents the original image, and $\mathcal{Y}_{rec} \in \mathbb{R}^{I_{1} \times I_{2} \times I_{3}}$ denotes the reconstructed image.

Figure \ref{tiffany_figure} presents the results of the reconstruction using $20\%$ of the available pixels from the three images (Girl, Sweet peppers, and fruits). We use 627 patches from given images, each of size $20\times 20 \times 3$, to train a dictionary of size $20\times 24\times 20 \times 3$, We then compare our results with the TNN and RTC methods.

\begin{figure}[h]
	\includegraphics[scale=0.34]{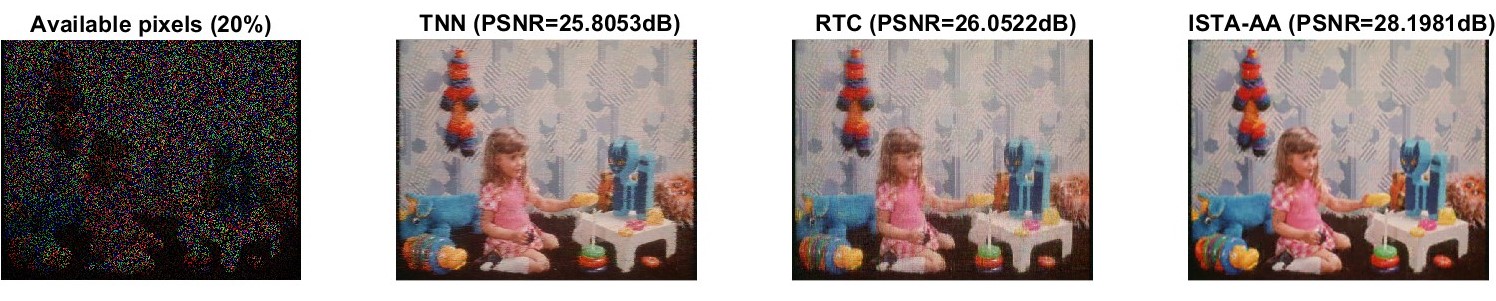}
     \includegraphics[scale=0.34]{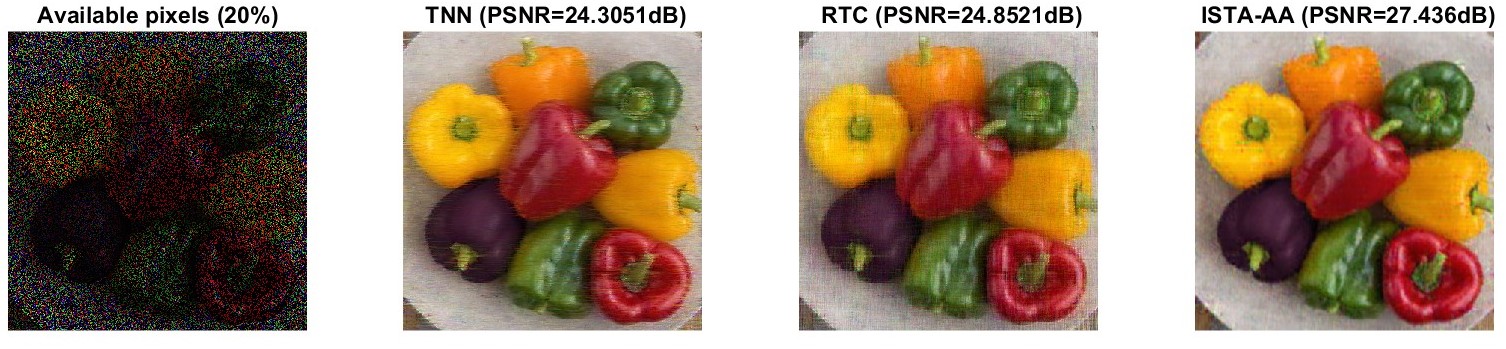}
     \includegraphics[scale=0.34]{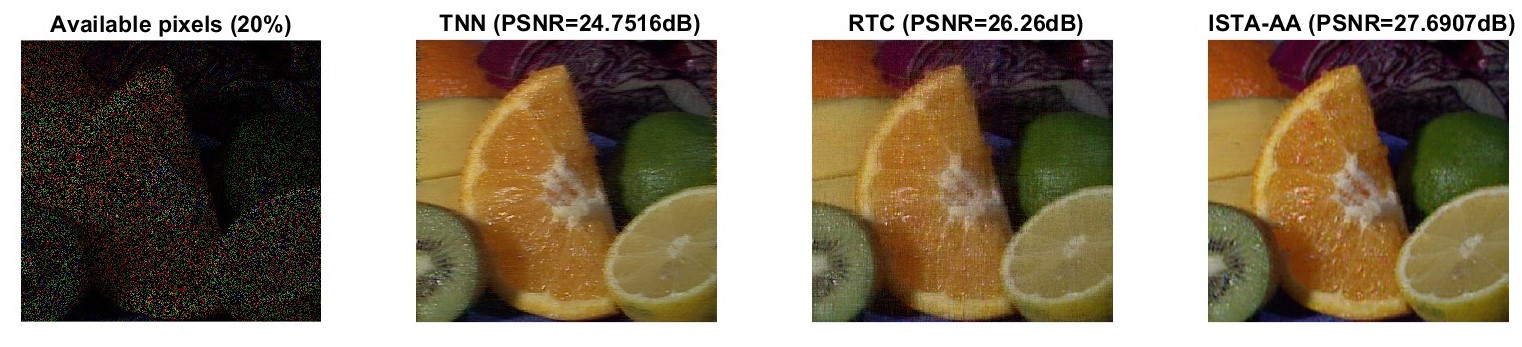}
	\caption{Illustrates the performance of the method on the completion problem applied to Tiffany, Lena, and Peppers images, with  $20 \%$ and of the pixels available.}
	\label{tiffany_figure}
\end{figure}

\subsection{Validation of the algorithm \ref{AISTA}}
In this experiment, we demonstrate the efficiency of Algorithm \ref{AISTA} compared to other acceleration methods such as FISTA, which uses Nesterove acceleration, and the TMPE  acceleration method \cite{bentbib_2025}. We use the same experimental setup for the images of Lena and Tiffany as in the previous example. In the reconstruction step, we apply the three methods—FISTA, ISTA-AA, and ISTA-TMPE and compare their results.
The error is given by:
\[error=\dfrac{\Vert\mathcal{Y}_{org}-\mathcal{Y}_{rec}\Vert_{F}}{\Vert\mathcal{Y}_{org}\Vert_{F}}.\]

Figures \ref{curve figure1} and \ref{curve figure2} show the error function plotted against iterations for the three accelerated algorithms—FISTA, ISTA-AA, and ISTA-TMPE—when applied to two different images.

 ISTA-AA achieves higher precision than both FISTA and ISTA-TMPE, particularly as the number of available pixels is reduced.
\begin{figure}[h]
	\includegraphics[scale=0.3]{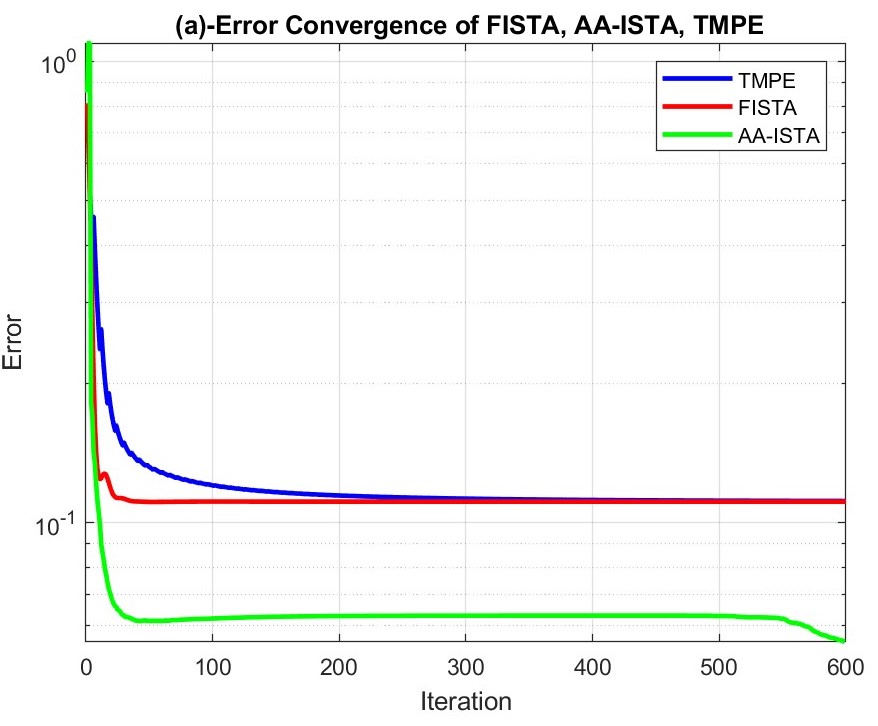}
	\includegraphics[scale=0.3]{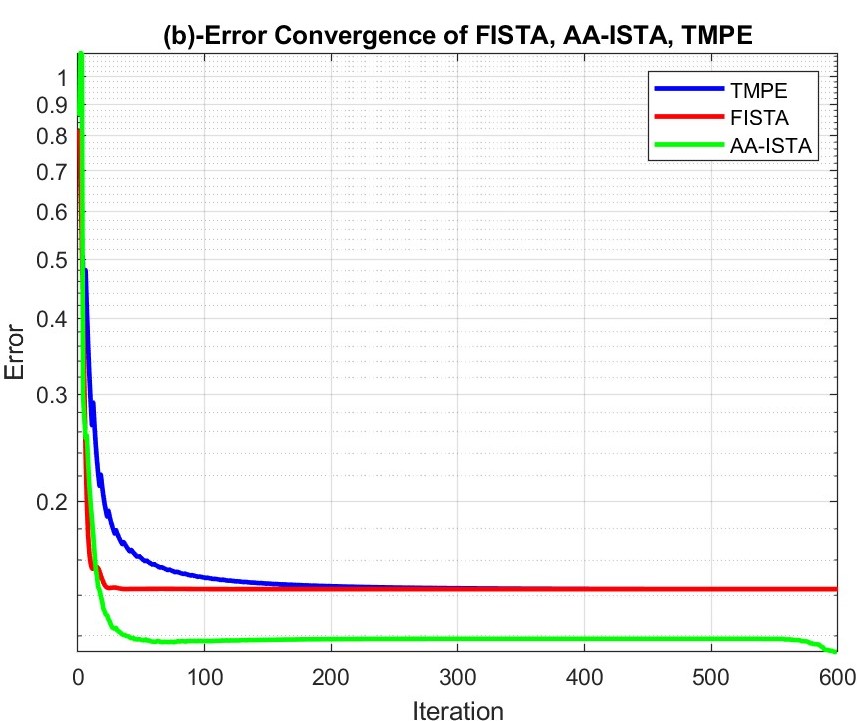}

	\caption{(a)-(b) -shows the error values as a function of iterations for the three methods applied to the completion problem of two different images with $30\%$ of pixels available.}
	
	\label{curve figure1}
\end{figure}
\begin{figure}[h]
	\includegraphics[scale=0.3]{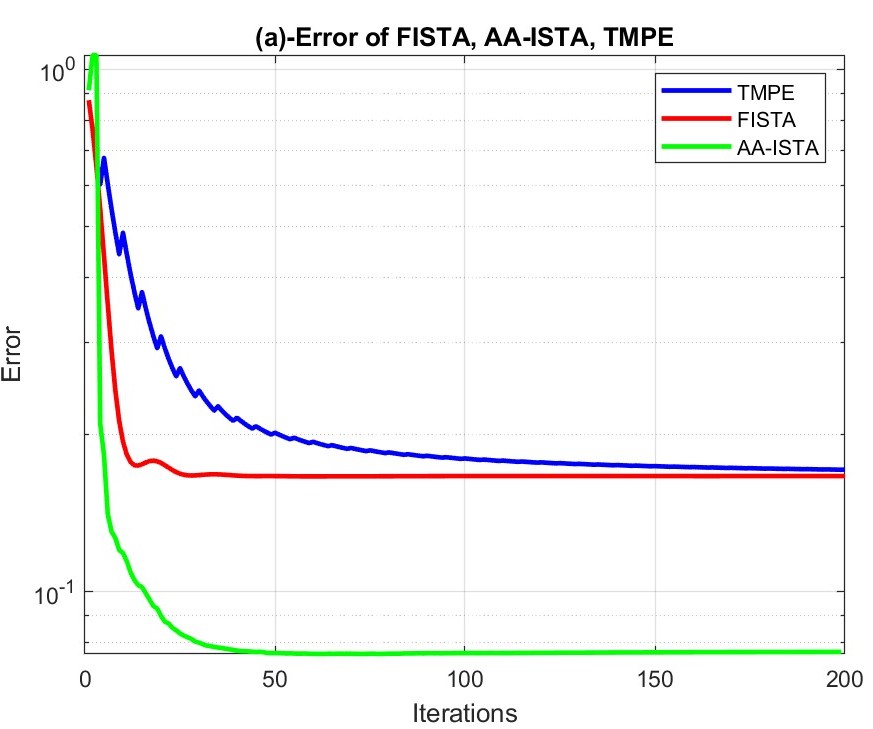}
	\includegraphics[scale=0.3]{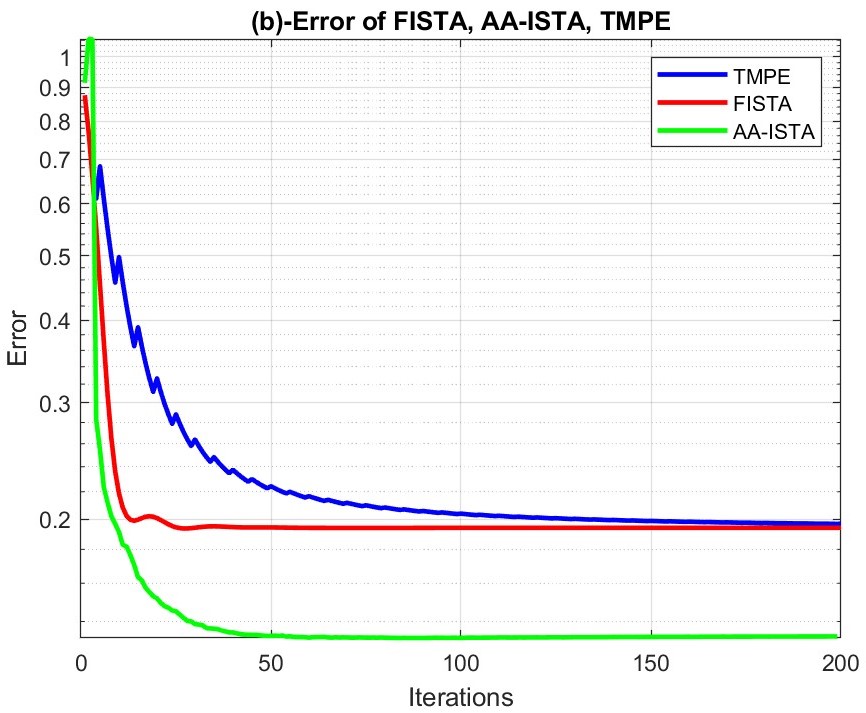}
	
	\caption{((a)-(b) -shows the error values as a function of iterations for the three methods applied to the completion problem of two different images with $20\%$ of pixels available.}
	
	\label{curve figure2}
\end{figure}
\subsection{Comparison between PGSD and the second order method }

In this experiment, we compare dictionaries generated by PGSD  and the second order method using ISTA. We initialize the dictionary atoms using random patches from the training set, and we use algorithms \ref{ODL2} and \ref{ODL} to train two dictionaries from 300 patches of size  $12\times 12\times 3$ and $20\times 20\times 3$  respectively from peppers image.In PGSD algorithm we choose the parameters $\eta^{(t)}=\dfrac{a}{b+t}$ with $a=10$, $b=5$. Figure \ref{curve figure4} shows the results of the learning process. The table above presents the results of image reconstruction using different dictionaries 

\begin{figure}[h]
	\includegraphics[scale=0.18]{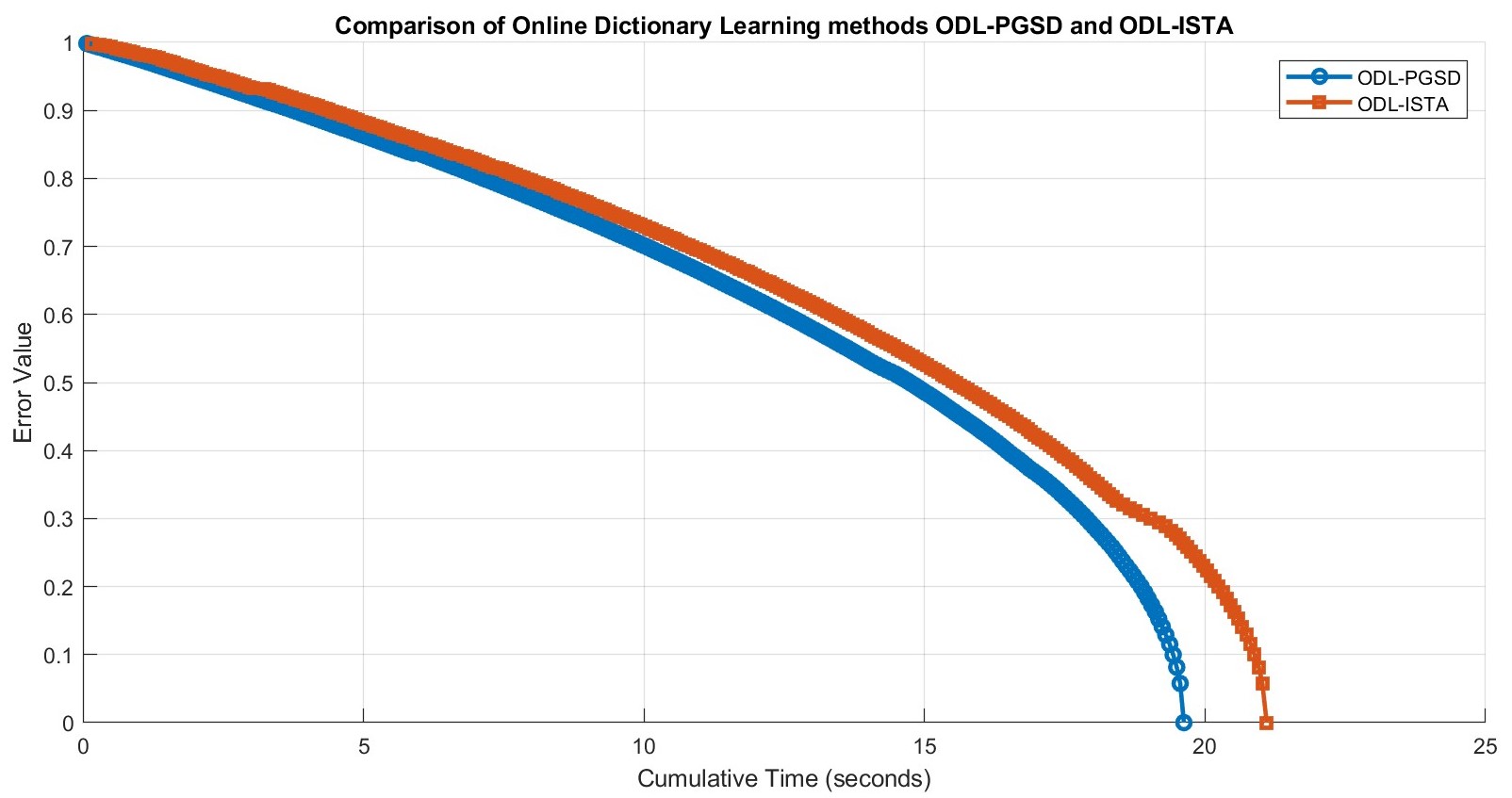}
	\includegraphics[scale=0.18]{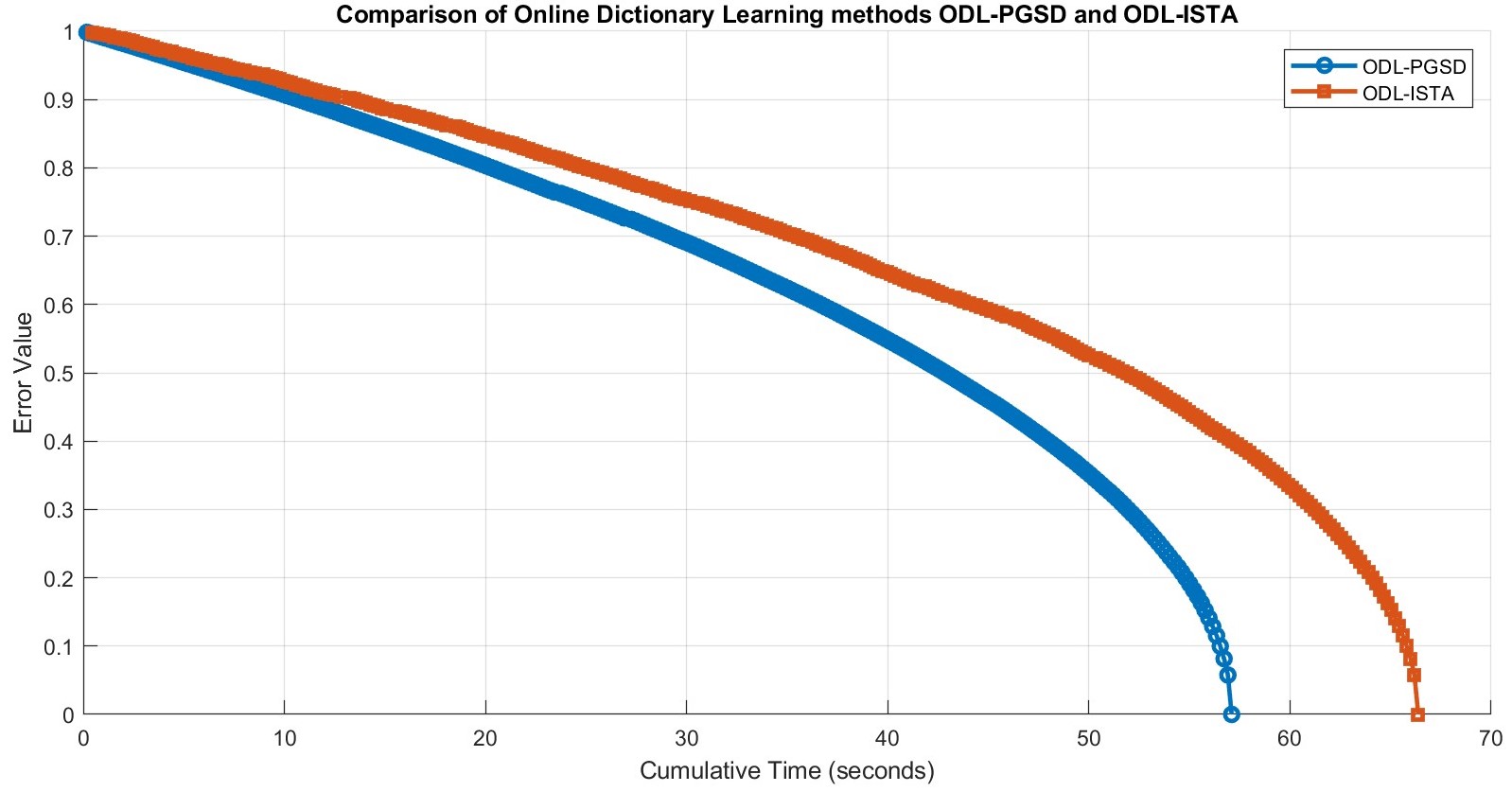}
	
	\caption{(a)-(b) -shows the error values as a function of time for the online dictionary learning methods PGSD and ISTA.}
	
	\label{curve figure4}
\end{figure}

\begin{tabular}{|p{2cm}|p{3cm}|p{2cm}|p{2cm}|p{2cm}|p{2cm}|}
	\hline
	Method &  size of patches & number of patches & Learning time in (s)  &  Error & PSNR  \\
	\hline
	PGSD  &$12\times 12\times 3$ & 450 & 78.15 & 0.13 & 22.4 \\ 
	  & $20\times 20\times 3$ &200 &35.74 &0.086 & 26.15\\
	  & $20\times 20\times 3$ & 450 & 92.92 &0.07 &27.46 \\
	\hline
   ISTA  & $12\times 12\times 3$ & 450 & $85.16$ & 0.10  &  24.51  \\
	& $20\times 20\times 3$&200 & 38.88& 0.085& 26.31\\
	& $20\times 20\times 3$  &450  & 95.98 &0.06 &28.24 \\
	\hline
\end{tabular} \\

Clearly, from the table above and Figure \ref{curve figure4}, the online projected stochastic gradient descent (PSGD) demonstrates a shorter training time compared to ISTA. However, we can observe that dictionary learning using ISTA yields slightly better PSNR and lower error in image reconstruction.

\section{Conclusion}\label{sec8}
We have proposed a more generalized model for dictionary learning using the t-product of p-order tensors. Additionally, we generalized two online algorithms that demonstrate higher performance compared to previous batch methods. To further enhance performance, we accelerated the ISTA algorithm by applying a generalized form of Anderson acceleration, which outperforms other acceleration methods in the context of completion. We demonstrated the effectiveness of our methods through several numerical experiments, particularly when handling multidimensional data.
\section*{Statements and Declarations}

\subsection*{Competing Interests}
The authors declare no competing interests that could have influenced the work presented in this paper.

\subsection*{Funding}
This research was funded by the Moroccan Ministry of Higher Education, Scientific  Research and Innovation and the OCP Foundation through the APRD research program.

\begin{appendices}




\end{appendices}

\bibliography{bibliography-paper}


\end{document}